\documentclass[a4paper,12pt]{article}
\usepackage{}
\usepackage{mathrsfs}
\usepackage{amssymb}
\usepackage{amsmath}
\usepackage{amsmath,amssymb,amsthm,graphics, graphicx,latexsym,amsfonts}
\usepackage{fancyhdr}
\usepackage{color}

\title{Flow polynomials of a signed graph}
\author{Jianguo Qian\footnote{email address: jgqian@xmu.edu.cn}\\
\small  School of Mathematical Sciences,  Xiamen University\\
\small  Xiamen, Fujian 361005, P.R. China }
\date{}
\usepackage{indentfirst}
\begin{document}
\maketitle
\newtheorem{lem}{Lemma}[section]
\newtheorem{thm}[lem]{Theorem}
\newtheorem{prop}[lem]{Proposition}
\newtheorem{cor}[lem]{Corollary}
\newtheorem*{pf}{Proof}
\begin{abstract}
In contrast to ordinary graphs, the  number of the nowhere-zero group-flows in a signed graph may vary with different groups, even if the groups have the same order. In fact, for a signed graph $G$ and non-negative integer $d$, it was shown that there exists a polynomial $F_d(G,x)$ such that the number of the nowhere-zero $\Gamma$-flows in $G$ equals $F_d(G,x)$ evaluated at $k$ for every Abelian group $\Gamma$ of order $k$ with $\epsilon(\Gamma)=d$, where $\epsilon(\Gamma)$ is the largest integer $d$ for which $\Gamma$ has a subgroup isomorphic to $\mathbb{Z}^d_2$. We focus on the combinatorial structure of  $\Gamma$-flows in a signed graph and the coefficients in $F_d(G,x)$. We first define the fundamental directed circuits for a signed graph $G$ and show that all $\Gamma$-flows (not necessarily nowhere-zero) in $G$ can be generated by these circuits. It turns out that all $\Gamma$-flows in $G$ can be evenly classified into $2^{\epsilon(\Gamma)}$-classes specified by the elements of order 2 in $\Gamma$, each class of which consists of the same number of flows depending only on the order of the group. This gives an explanation for why the number of $\Gamma$-flows in a signed graph varies with different $\epsilon(\Gamma)$, and also gives an answer to a problem posed by Beck and Zaslavsky. Secondly, using an extension of  Whitney's broken circuit theory we give a combinatorial interpretation of the coefficients in $F_d(G,x)$ for $d=0$, in terms of the broken bonds. As an example, we give an analytic expression of $F_0(G,x)$ for a class of the signed graphs that contain no balanced circuit. Finally, we show that the sets of edges  in a signed graph that contain no broken bond form a  homogeneous  simplicial complex.
\end{abstract}
\noindent\textbf{Keywords:} signed graph; nowhere-zero flow; polynomial; coefficient\\
\noindent\textbf{AMS classification:} 05C21; 05C22; 05C31

\section{Introduction}
Nowhere-zero  $\mathbb{Z}_k$-flows, or modular $k$-flows, in a graph was initially introduced by Tutte \cite{Tutte02} as a dual problem to vertex-colouring of plane graphs. It has long been known that the number of nowhere-zero $\mathbb{Z}_k$-flows or more in general,  nowhere-zero $\Gamma$-flows (flows with values in $\Gamma$) for an Abelian group $\Gamma$ of order $k$ is a polynomial function in $k$, which does not depend on the algebraic structure of the group \cite{Tutte02}. An analog to $\mathbb{Z}_k$-flow  is the integer $k$-flow or, simply $k$-flow. It is well known that  a graph has a nowhere-zero $k$-flow if and only if it has a nowhere-zero $\mathbb{Z}_k$-flow \cite{Tutte01}. In \cite{Kochol}, Kochol showed that the number of nowhere-zero $k$-flows is also a polynomial in $k$, although not the same polynomial as that for nowhere-zero $\mathbb{Z}_k$-flows. For more topics related to nowhere-zero flows in graphs see also Brylawski and Oxley \cite{Brylawski03}, Jaeger \cite{Jaeger}, Seymour \cite{Seymour} and Zhang \cite{Zhang}.

The notion of the signed graphs was introduced by Harary \cite{Harary} initially as a model for social networks. In comparison with flows in plane graphs or more generally, in graphs embedded on orientable surface, the definition of $\mathbb{Z}_k$-flows in signed graphs is naturally considered for the study of graphs  embedded on non-orientable surface, where nowhere-zero  $\mathbb{Z}_k$-flows emerge as the dual notion to local tensions \cite{Kaiser}.

In contrast to ordinary graphs, the problem of counting the nowhere-zero flows in a signed graph seems more complicated and there are relatively few results to be found in the literatures. Using the theory of counting lattice points in inside-out polytopes to signed graphs, Beck and Zaslavsky  \cite{beck01} showed that the number of the nowhere-zero $k$-flows in a signed graph is a quasi-polynomial of period two, that is, a pair of polynomials, one for odd values of $k$ and the other for even $k$.  In the same paper, Beck and Zaslavsky also showed that there exists a polynomial $f(G,x)$ such that, for every odd integer $k$, the number of nowhere-zero $\Gamma$-flows in a signed graph $G$ equals $f(G,x)$ evaluated at $k$ for every Abelian group $\Gamma$ with $|\Gamma|=k$. This result was recently  extended by DeVos, Rollov\'{a} and \v{S}\'{a}mal \cite{DeVos} (available from arXiv) to general Abelian group: for any non-negative integer $d$, there exists a polynomial $f_d(G,x)$ such that the number of nowhere-zero $\Gamma$-flows in $G$ is  exactly $f_d(G,x)$ evaluated at $n$ for every Abelian group $\Gamma$ with $\epsilon(\Gamma)=d$ and $|\Gamma|=2^dn$, where $\epsilon(\Gamma)$ is the largest integer $d$ for which $\Gamma$ has a subgroup isomorphic to $\mathbb{Z}^d_2$.  More recently,  Goodall et. al. \cite{Goodall} (available from arXiv) gave an explicit expression of $f_d(G,x)$  in form of  edge-subgraph expansions.

In this paper we focus on the combinatorial structure of  $\Gamma$-flows in a signed graph $G$ and the coefficients in the polynomial $f_d(G,x)$. For convenience,  instead of working on $f_d(G,x)$,  we will work on  the polynomial $F_d(G,x)$ defined by $F_d(G,x)=f_d(G,2^{-d}x)$ and call $F_d(G,x)$ the {\it $d$-type flow polynomial}, or simply, the {\it flow polynomial} of $G$. It can be seen that $F_d(G,x)$ evaluated at $k$  is  exactly the number of the nowhere-zero $\Gamma$-flows in $G$ for every Abelian group $\Gamma$ with $\epsilon(\Gamma)=d$ and $|\Gamma|=k$.

In the third section we introduce the fundamental directed circuits  and the fundamental  root circuit (a particular unbalanced circuit) in a signed graph $G$. We show that every $\Gamma$-flow (not necessarily nowhere-zero) in $G$ can be generated by these circuits, each of which is assigned with a proper $\Gamma$-flow. More specifically, the values of the flows assigned to the fundamental directed circuits are the elements in $\Gamma$ while the value to the fundamental  root circuit is an element of order 2 in the group $\Gamma$. Therefore, all $\Gamma$-flows in $G$ can be  evenly classified into $2^{\epsilon(\Gamma)}$-classes specified by  the elements of order 2 in $\Gamma$. Moreover, each class consists of the same number of flows, which depends only on the order of the group.  This gives an explanation for why the number of the $\Gamma$-flows in a signed graph  varies with different $\epsilon(\Gamma)$ and, also gives an answer to a problem posed by Beck and Zaslavsky in \cite{beck01}. Further, this result also yields an explicit expression of the polynomial $F_d(G,x)$ obtained earlier by  Goodall et. al.

 In the fifth section we give a combinatorial interpretation of the coefficients in $F_d(G,x)$ for $d=0$. To this end, we apply Whitney's broken circuit theory \cite{Whitney01}.  In the study of graph coloring,  one significance of Whitney's broken circuit theorem is that it gives a very nice  `cancellation' to reduce the terms in the chromatic polynomial (represented in the form of inclusion-exclusion) so that the remaining terms can not be cancellated out anymore and,
therefore, yield a combinatorial interpretation for the coefficients of the polynomial \cite{Brylawski01,Brylawski02}. Using an extended form of the Whitney's theorem given by Dohmen and Trinks \cite{Dohmen01}, we show that $F_0(G,x)$ is a polynomial with leading term $x^{m-n}$ and with its coefficients alternating in signs. More specifically, the coefficient of $(-1)^{i}x^{m-n-i},i=0,1,\cdots,m-n$, is exactly the number of the sets consisting of $i$ edges  that contain no broken bond.  As an example, we give an analytic expression of $F_0(G,x)$ for a class of the signed graphs that contain no balanced circuit. Finally, we show that the broken bonds in a signed graph form a nice topological structure, that is, a homogeneous simplicial complex of top dimension $m-n$. Thus, the coefficients of $F_0(G,x)$
are the simplex counts in each dimension of the complex.

\section{Preliminaries}

Graphs in this paper may contain parallel edges or loops.  For a graph $G$, we use $V(G)$ and $E(G)$ to denote its vertex set and edge set, respectively.  A {\it signed graph} is a pair $(G,\Sigma)$, where $\Sigma\subseteq E(G)$ and the edges in $\Sigma$ are negative while the other ones are positive.

A {\it circuit} is a connected 2-regular graph. An {\it unbalanced circuit} in a signed graph $(G, \Sigma)$ is a  circuit in $G$ that has an odd number of negative edges. A {\it balanced circuit} in $(G, \Sigma)$ is a  circuit that is not unbalanced. A subgraph of $G$ is {\it unbalanced} if it contains an unbalanced circuit; otherwise, it is {\it balanced}.  In particular, a subgraph without negative edges is balanced. A {\it barbell} is the union of two unbalanced circuits $C_1,C_2$ and a (possibly trivial) path $P$ with end vertices $v_1\in V(C_1)$ and $v_2\in V(C_2)$, such that
$C_1-v_1$ is disjoint from $P\cup C_2$ and $C_2-v_2$ is disjoint from $P\cup C_1$. We call $P$ the {\it barbell path} of the barbell. A {\it signed circuit} is either a balanced circuit or a barbell.

Given a signed graph $(G,\Sigma)$, {\it switching} at a vertex $v$ is the inversion of the sign of each edge incident with $v$. Two signed graphs are said to be  {\it switching-equivalent} if one can be obtained from the other by a series of switchings. It is known \cite{Kaiser} and easy to see that equivalent signed graphs have the same sets of unbalanced circuits and the same sets of balanced circuits. This means, in particular, that a balanced signed graph $(G,\Sigma)$ is switching-equivalent to an ordinary graph $G$.

Following Bouchet \cite{Bouchet} we now introduce the notion of the {\it half-edges} so as to orient
a signed graph: each negative edge of $(G,\Sigma)$ is viewed as composed of two half-edges. An orientation of a negative edge $e$ is obtained by giving each of the two half-edges $h$ and $h'$ a direction so that both $h$ and $h'$ point toward the end vertices of $e$, called  {\it extroverted}, or both $h$ and $h'$ point toward the inside of $e$, called {\it introverted}.

In the following, we will use $G$ simply to denote a signed graph if no confusion can occur. Let $D$ be a fixed orientation of a signed graph $G$ and $\Gamma$ be an additive Abelian group. A map ${\bf f}:E(D)\rightarrow \Gamma$ is called a $\Gamma$-flow  if   the usual conservation law (Kirchhoff's law) is satisfied, that is, for each vertex $v$, the sum of ${\bf f}(e)$ over the incoming edges $e\in E^-(v)$ at $v$ (i.e., the edges and half-edges directed towards $v$) equals the sum of ${\bf f}(e)$ over the outgoing edges $e\in E^+(v)$, i.e.,
$$\sum_{e\in E^+(v)}{\bf f}(e)=\sum_{e\in E^-(v)}{\bf f}(e).$$

 A flow ${\bf f}$ is called nowhere-zero if ${\bf f}(e)\not=0$ for each $e\in E(D)$. It is well known that the number of nowhere-zero
$\Gamma$-flows is independent of  the orientation of $G$. A signed graph is said to be $\Gamma$-{\it flow admissible}  if it admits at least one nowhere-zero $\Gamma$-flow. It is clear that the property of `$\Gamma$-flow admissible' is invariant under switching inversion.

\section{Fundamental  circuits in a signed graph}

In this section we generalize the notion of fundamental circuits in graphs to signed graphs, which will play an important role in revealing the structural property of $\Gamma$-flows in signed graphs.

For a signed graph $G$ and a set $F$ of edges, we denote by $G+F$ and $G-F$ the subgraphs obtained from $G$ by adding and deleting the edges in $F$, respectively. Let $E_N=\{e_0,e_1,e_2,\cdots,e_{m_N-1}\}$ be the set of all the negative edges of $G$, where $m_N=|E_N|$. In this section we always assume that $G$ is unbalanced and, with no loss of generality, contains as few negative edges as possible in its switching equivalent class. Thus, $E_N\not=\emptyset$ and $G-E_N$ is connected \cite{Zaslavsky}.

Let $T$ be a spanning tree of $G-E_N$. Choose an arbitrary edge $e_0$ from $E_N$ and call $T_{0}=T+e_0$  a {\it signed rooted tree} of $G$ with root edge $e_0$ (note that a signed rooted tree we defined here is not a real tree because it has a unique unbalanced circuit). Let $\overline{T}_0=E(G)\setminus E(T_0)$. For any $e\in \overline{T}_0$, it is clear that $T_{0}+e$ contains a unique signed circuit. We call this circuit a {\it fundamental circuit} and denote it by $C_e$. We can see that, if $e\in\overline{T}_0\setminus E_N$ then $C_e$ is an ordinary circuit (a circuit without negative edge) and if  $e\in E_N\setminus\{e_0\}$ then $C_e$ is a barbell or a balanced circuit with two negative edges $e_0$ and $e$.

 Given a fixed orientation $D$ of $G$, a {\it fundamental directed circuit} $\overrightarrow{C}_e$ of $G$ is the orientation of a fundamental circuit $C_e$ such that the direction of $e$ is the same as what it was in $D$ and  the directions of all other edges on $\overrightarrow{C}_e$ coincide consistently with $e$ along with $C_e$. Under this orientation, it can be seen that if $C_e$ is an ordinary circuit then an edge $e'$ on  $\overrightarrow{C}_e$ is clockwise oriented if and only $e$ is  clockwise oriented, and if $C_e$ is a balanced circuit or a barbell (with two negative edges $e_0$ and $e$), then the direction of the two negative edges are always opposite, that is, $e_0$ is extroverted if and only if $e$ is introverted, see Figure 1.
\begin{figure}[htbp]
\begin{center}
\includegraphics[height=4.6cm]{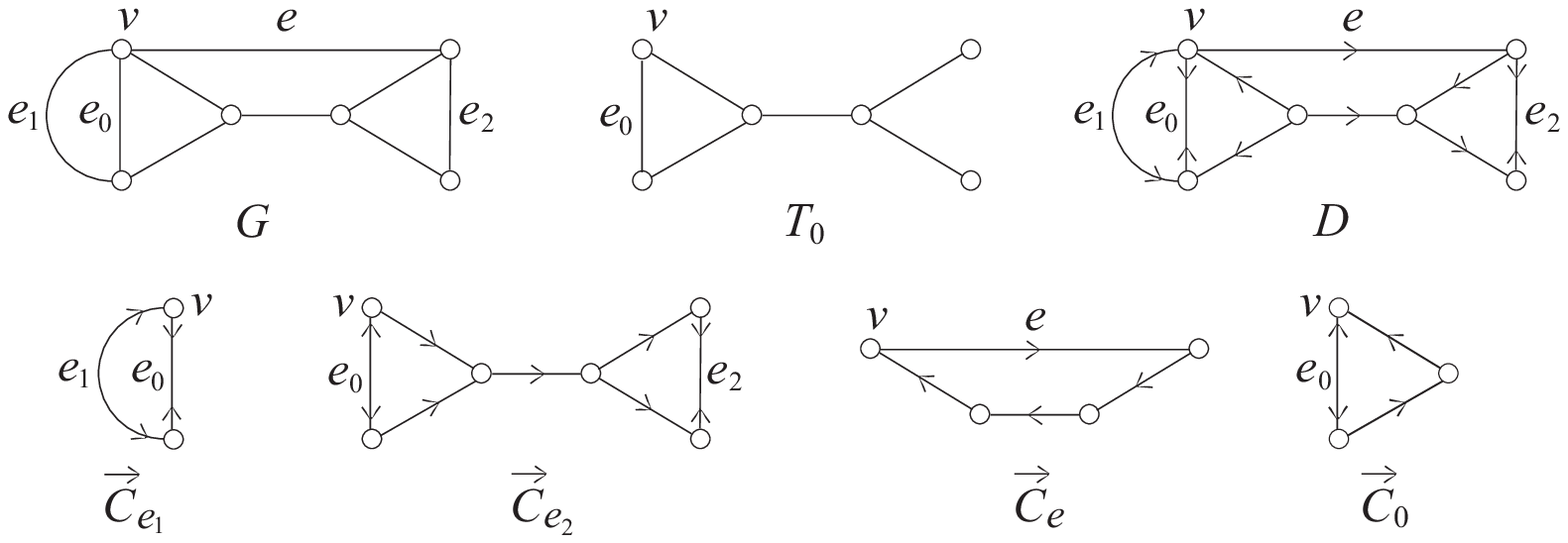}
{\bf Figure 1.}\,\ The edges $e_0,e_1,e_2$ are negative and $e$ is positive.
\end{center}
\end{figure}

 For an fundamental circuit $C_e$, let $C^D_e$ be the orientation $D$ restricted on $C_e$. We associate with ${C}_e$ a function ${\bf f}_{e}$ on $E(G)$ defined by
$${\bf f}_{e}(a)=\left\{\begin{array}{rll}
1,&& {\rm if}\  a\in \overrightarrow{C}_e;\\
-1,&& {\rm if}\  a\in C_e^D\setminus \overrightarrow{C}_e;\\
2,&& {\rm if}\  a\in \overrightarrow{C}_e\ {\rm and}\ a\ {\rm is\ on\ the\ barbell\ path\ of}\ C_e;\\
-2,&& {\rm if}\  a\in C_e^D\setminus \overrightarrow{C}_e\ {\rm and}\ a\ {\rm is\ on\ the\ barbell\ path\ of}\ C_e;\\
0,&& {\rm otherwise}
\end{array}\right.$$
for any $a\in E(D)$, where `$a$ is on the barbell path of $C_e$' means that $C_e$ is a barbell and $a$ is on the barbell path of $C_e$.

Form the above definition, it can be seen that ${\bf f}_e(e)=1$ for any $e\in \overline{T}_0$.

Let $C_0$ be the unique (un-balanced) circuit in $T_0$ (i.e., formed by $e_0$ and $T$). Choose an arbitrary vertex $v$ on $C_0$ and let $\overrightarrow{C}_0$ be the orientation of $C_0$ such that the direction of $e_0$ is extroverted  and all other edges on  $C_0$ are oriented so that $d^-(v)=2,d^+(v)=0$ and $d^-(u)=d^+(u)=1$ for any vertex $u$ on $C_0$ other than $v$,  where $d^-(v)$ and $d^+(v)$ are the in-degree and out-degree of $v$ on  $\overrightarrow{C}_0$, respectively, see Figure 1. We call $\overrightarrow{C}_0$ the {\it fundamental  root circuit} and associate it with a function ${\bf g}$ on $E(G)$ defined by
$${\bf g}(e)=\left\{\begin{array}{rll}
1,&& {\rm if}\  e\in \overrightarrow{C}_0;\\
-1,&& {\rm if}\  e\in C^D_0\setminus\overrightarrow{C}_0;\\
0,&& {\rm otherwise}
\end{array}\right.$$
for any $e\in E(D)$.

 For convenience, in the following we regard each $\Gamma$-flow, each function ${\bf f}_e$ ($e\in\overline{T}_0$) and the function ${\bf g}$ as $m$-dimensional vectors indexed by $e\in E(G)$. Let ${\cal S}_G$ denote the class of all $\Gamma$-flows (not necessarily nowhere-zero) in $G$.

For a finite additive Abelian group $\Gamma$, let $\Gamma_2$ be the set of the elements of order 2 in $\Gamma$ (including the zero element). Recalling that $\epsilon(\Gamma)$ is the largest integer $d$ for which $\Gamma$ has a subgroup isomorphic to $\mathbb{Z}^d_2$, we have $|\Gamma_2|=2^{\epsilon(\Gamma)}$.

\begin{thm}\label{basis} Let $\Gamma$ be an additive Abelian group and let $G$ be a connected unbalanced signed graph. Let $T$ be a spanning tree $T$ of  $G$ consisting of positive edges and let $e_0\in E_N$. Then\\
\begin{equation}
{\cal S}_G=\{\gamma{\bf g}+\sum_{e\in \overline{T}_0}\gamma_e{\bf f}_{e}:\gamma\in\Gamma_2,\gamma_e\in\Gamma\}.
\end{equation}
 \end{thm}
\begin{proof}  It is clear that
\begin{equation}
\gamma{\bf g}+\sum_{e\in \overline{T}_0}\gamma_e{\bf f}_{e}
 \end{equation}
 is a $\Gamma$-flow for any $\gamma\in\Gamma_2$ and $\gamma_e\in\Gamma$. Let ${\bf f}$ be an arbitrary $\Gamma$-flow in $G$. We need only prove that  ${\bf f}$ can be written as the combination  (2).

Since a $\Gamma$-flow is independent of the orientation $D$, to simplify our discussion we make the following assumption:

\noindent{\bf Assumption 1}. In orientation $D$, the direction of  the root edge $e_0$ is extroverted while the directions of all other negative edges are introverted.

 For each negative edge $e_i=u_iv_i\in E_N$, insert a new vertex $w_i$ into the middle of $e_i$ so that the two half edges of $e_i$ in $D$ become two ordinary directed edges $w_iu_i$ (with direction from $w_i$ to $u_i$) and $w_iv_i$  if $i=0$, or $u_iw_i$ and $v_iw_i$ if $i\in\{1,2,\cdots,m_N-1\}$. We call $w_i$ the {\it middle vertex} of $e_i$.

Further, add a new vertex $w$ to $D$ and, for each middle vertex $w_i$, add the directed edge $e'_i=ww_i$ if $i=0$ and the directed edge $e'_i=w_iw$ if $i\in\{1,2,\cdots,m_N-1\}$. The resulting graph, denoted by $D^w$,  is a directed graph without negative edges, that is, $D^w$ is an ordinary directed graph. Further, for $i\in\{0,1,2,\cdots,m_N-1\}$, assign the edge $e'_i$ with the value $2{\bf f}(e_i)$.
It clear that, except the possible $w$, the conservation law is satisfied at all the vertices in $D^w$ and therefore must be satisfied at $w$,  either. As a result, we get a span of the $\Gamma$-flow ${\bf f}$ to $D^w$ and denote it by ${\bf f}^w$. Thus, by the conservation law at $w$, we have
\begin{equation*}
{\bf f}^w(e'_0)=\sum_{i=1}^{m_N-1}{\bf f}^w(e'_i)
\end{equation*}
or equivalently,
\begin{equation}
2{\bf f}(e_0)=\sum_{i=1}^{m_N-1}2{\bf f}(e_i)=2\sum_{e_i\in E^*_N}{\bf f}(e_i),
\end{equation}
where $E^*_N=E_N\setminus\{e_0\}=\{e_1,e_2,\cdots,e_{m_N-1}\}$.

Further,  we notice that, for any $\gamma\in\Gamma$, the solution of the equation $2x=2\gamma$ (in $x$) over $\Gamma$ has the form $x=\gamma+\gamma_2$, where  $\gamma_2$ is an element of order 2 (possibly the zero element), i.e., $\gamma_2\in \Gamma_2$.  Thus, (3) is equivalent to
\begin{equation}
{\bf f}(e_0)=\gamma_2+\sum_{e_i\in E^*_N}{\bf f}(e_i),
\end{equation}
where $\gamma_2\in \Gamma_2$.

 On the other hand, for any $e\in E^*_N$, by Assumption 1 and the  definitions of $\overrightarrow{C}_e$  and ${\bf f}_e$,  we have
\begin{equation}
{\bf f}_e(e_0)={\bf f}_e(e_i)=1.
\end{equation}

In (2), we set $\gamma=\gamma_2$ and for $e\in \overline{T}_0$, set $\gamma_e={\bf f}(e)$. Let
\begin{equation}{\bf f'}={\bf f}-(\gamma_2{\bf g}+\sum_{e\in \overline{T}_0}\gamma_e{\bf f}_{e}).
 \end{equation}
 Then for any $e\in \overline{T}_0$, by the definition of the vector ${\bf g}$ we have $\gamma_2{\bf g}(e)=0$ since $e$  is not on $C_0$. This implies that ${\bf f}'(e)=0$ for any $e\in \overline{T}_0$ because $\gamma_e={\bf f}(e)$ and, as mentioned earlier, ${\bf f}_e(e)=1$. Further, by (4), (5) and (6) we have
$$\begin{array}{rll}
{\bf f}'(e_0)&=&{\bf f}(e_0)-(\gamma_2{\bf g}(e_0)+\sum\limits_{e\in \overline{T}_0}\gamma_e{\bf f}_e(e_0))\\
&=&\gamma_2+\sum\limits_{e\in E^*_N}{\bf f}(e)-(\gamma_2{\bf g}(e_0)+\sum\limits_{e\in \overline{T}_0\setminus E_N^*}\gamma_e{\bf f}_e(e_0)+\sum\limits_{e\in E_N^*}\gamma_e{\bf f}_e(e_0))\\
&=&\sum\limits_{e\in E^*_N}{\bf f}(e)-\sum\limits_{e\in E_N^*}\gamma_e{\bf f}_e(e_0)\\
&=&\sum\limits_{e\in E^*_N}{\bf f}(e)(1-{\bf f}_e(e_0))\\
&=&0,
\end{array}$$
where the third equality holds because ${\bf g}(e_0)=1$ and $e_0\notin C_e$  for any $e\in \overline{T}_0\setminus E_N^*$ and therefore,  ${\bf f}_e(e_0)=0$; and the last two equalities hold because of (5) and $\gamma_e={\bf f}(e)$ for any $e\in E^*_N$.

The above discussion means that ${\bf f'}$  evaluated at each edge outside of $T$ is zero.   Thus, we must have ${\bf f'}={\bf 0}$  (the vector of all zeros) because the values of ${\bf f}'$ at the edges of $T$ are uniquely determined by that  outside of $T$. In conclusion, ${\bf f}$ is represented as the combination (2), which completes our proof.
\end{proof}

\section{Classification of $\Gamma$-flows in a signed graph}

 From Theorem \ref{basis}, we have known that all the $\Gamma$-flows in a connected unbalanced signed graph can be `generated' by fundamental root circuit $\overrightarrow{C}_0$ and the fundamental directed circuits $\overrightarrow{C}_e,e\in \overline{T}_0$. This leads to the following classification of $\Gamma$-flows in a signed graph, which are specified by the elements of order 2 in $\Gamma$.

\begin{thm}\label{partition} Let $\Gamma$ be an additive Abelian group of order $k$ and let $G$ be a connected unbalanced signed graph. Let $T$ be a spanning tree $T$ of  $G$ consisting of positive edges and let $e_0\in E_N$.\\
1). The flows in ${\cal S}_G$ are pairwise distinct and, therefore,
\begin{equation}
|{\cal S}_G|=2^{\epsilon(\Gamma)}k^{m-n};
\end{equation}
2). ${\cal S}_G$ can be evenly classified into $|\Gamma_2|$ classes specified by the elements in $\Gamma_2$, i.e.,
${\cal S}_G=\bigcup_{\gamma\in\Gamma_2}{\cal S}_G(\gamma)\ \ \ {\rm and}\ \ \ |{\cal S}_G(\gamma)|=k^{m-n}\ {\rm for\ any}\ \gamma\in \Gamma_2,$
where
\begin{equation}
{\cal S}_G(\gamma)=\{\gamma{\bf g}+\sum_{e\in \overline{T}_0}\gamma_e{\bf f}_{e}:\gamma_e\in\Gamma\}.
\end{equation}
\end{thm}
\begin{proof} 1). We need only prove that
\begin{equation}
\gamma{\bf g}+\sum_{e\in \overline{T}_0}\gamma_e{\bf f}_{e}=\gamma'{\bf g}+\sum_{e\in \overline{T}_0}\gamma'_e{\bf f}_{e}
\end{equation}
if and only if $\gamma=\gamma'$ and $\gamma_e=\gamma'_e$ for any $e\in \overline{T}_0$. For any $e\in \overline{T}_0$, by the definition of {\bf g} and {\bf f}$_e$ we have ${\bf f}_e(e)=1,{\bf g}(e)=0$ and ${\bf f}_{e'}(e)=0$ for any $e'\in \overline{T}_0$ with $e'\not=e$. Thus, (9) implies that $\gamma_e{\bf f}_e(e)=\gamma_e'{\bf f}_e(e)$ and therefore, $\gamma_e=\gamma_e'$ for any $e\in \overline{T}_0$. Consequently, again by (9), we have $\gamma{\bf g}=\gamma'{\bf g}$ and therefore, $\gamma=\gamma'$.

\noindent 2). Since the flows in ${\cal S}_G$ are pairwise distinct, 2) follows directly.\end{proof}

 For a component $\omega$ of  a signed graph $G$, denote
\begin{equation}
\beta(\omega)=\left\{\begin{array}{ccl}
m(\omega)-n(\omega)+1,&&{\rm if}\ \omega\ {\rm is\ balanced};\\
m(\omega)-n(\omega),&&{\rm if}\ \omega\ {\rm is\ unbalanced},
\end{array}\right.
\end{equation}
where $m(\omega)$ and $n(\omega)$ are the number of edges and vertices in $\omega$, respectively. In general, we denote $\beta(G)=\sum\beta(\omega)$, where the sum is taken over all the components  $\omega$ of $G$. Let $\kappa(G)$ be the number of unbalanced components and $F^*(G,\Gamma)$ be the number of  $\Gamma$-flows (not necessarily nowhere-zero) in $G$.
\begin{cor}\label{main}  Let $G$ be a signed graph and let $\Gamma$ be an additive Abelian group of order $k$. Then
\begin{equation}
F^*(G,\Gamma)=2^{\kappa(G)\epsilon(\Gamma)}k^{\beta(G)}.
\end{equation}
\end{cor}
\begin{proof} If $G$ is not connected then $F^*(G,\Gamma)=\prod F^*(\omega,\Gamma)$, where the product is taken over all the components $\omega$ of $G$. We need only consider the case when $G$ is connected.

If  $G$ is unbalanced then (11) follows directly from (7). Now assume that  $G$ is balanced. Recall that a balanced signed graph is switching-equivalent to an ordinary graph. In this case it is known \cite{Jin} that  the number of $\Gamma$-flows (not necessarily nowhere-zero) in an ordinary graph is $k^{m-n+1}$, i.e., $F^*(G,\Gamma)=k^{m-n+1}$, where $m$ and $n$ are the numbers of edges and vertices in $G$, respectively. This agrees with (11) because $\kappa(G)=0$ and $\beta(G)=m-n+1$ when $G$ is balanced. The proof is completed. \end{proof}

\noindent{\bf Remark 1}. When $k$ (the order of $\Gamma$) is odd, Beck and Zaslavsky posed a problem (Problem 4.2, \cite{beck01}): Is there any significance to $F^*(G,\Gamma)$ evaluated at even natural numbers? By Theorem \ref{partition} and Corollary \ref{main} we can now give an answer to this problem. For simplicity, let's consider the case when $G$ is connected and unbalanced. Since $k$ is odd, we have $\epsilon(\Gamma)=0$ and therefore, $F^*(G,\Gamma)=k^{m-n}$. Thus, $F^*(G,\Gamma)$ evaluated at an even number $h$ equals $h^{m-n}$, which is exactly the number of the $\Gamma'$-flows  in $G$ divided by $2^{\epsilon(\Gamma')}$ for any group $\Gamma'$ of order $h$. More specifically, by Theorem \ref{partition}, $F^*(G,\Gamma)$ evaluated at $h$  equals the number of those $\Gamma'$-flows in $G$ which have the form
$${\bf f}=\gamma{\bf g}+\sum_{e\in \overline{T}_0}\gamma_e{\bf f}_{e},\ \ \gamma_e\in\Gamma',$$
where $\gamma$ is an arbitrary fixed element of order 2 in $\Gamma'$ (in particular we may choose $\gamma=0$). \hfill$\square$

For any $e\in E(G)$, the number of the $\Gamma$-flows in $G$ with value 0 at $e$ is clearly equal to $F^*(G-e,\Gamma)$. Notice that the flows counted by $F_d(G,x)$ are nowhere-zero. So by Corollary \ref{main} and the principle of inclusion-exclusion, we get the following expression of $F_d(G,x)$ obtained by Goodall et. al.:
\begin{cor} \cite{Goodall}\label{excl-incl} For any signed graph $G$ and non-negative integer $d$,
$$F_d(G,x)=\sum_{F\subseteq E}(-1)^{|F|}2^{\kappa(G-F)d}x^{\beta(G-F)}.$$
\end{cor}

We note that, if $G$ is an ordinary graph then $\kappa(G-F)=0$ for any $F\subseteq E(G)$. Therefore, Corollary \ref{excl-incl} generalizes the corresponding result for ordinary graph \cite{Dong03,Jin}.

\noindent{\bf Example}. By Corollary \ref{excl-incl}, if $G$ is the graph with two vertices joined by a negative edge and a positive edge then $F_d(G,x)=2^d-1$; if $G$ is the graph consisting of two negative loops at a vertex then $F_d(G,x)=2^dx-2^{d+1}+1$; and if $G$ is the graph consisting of a negative loop and a positive loop at a vertex then $F_d(G,x)=(2^d-1)(x-1)$.
\section{Coefficients in $F_0(G,x)$}
In this section we will give a combinatorial interpretation of the coefficients in $F_d(G,x)$ for $d=0$. We begin with the following extension of Whitney's broken theorem given by Dohmen and Trinks.
\begin{lem}\label{broken thm} \cite{Dohmen01} Let $P$  be a finite linearly ordered set,  $\mathscr{B}\subseteq 2^P\setminus\{\emptyset\}$ and $\Gamma$  be an additive Abelian group. Let $f:2^P\rightarrow \Gamma$ be a mapping such that, for any $B\in\mathscr{B}$ and $A\supseteq B$,
\begin{equation}
f(A)=f(A\setminus\{B_{\max}\}).
\end{equation}
 Then
\begin{equation}
\sum_{A\in 2^P}(-1)^{|A|}f(A)=\sum_{A\in2^P\setminus \mathscr{B}^*}(-1)^{|A|}f(A),
\end{equation}
where  $B_{\max}$ is the maximum element in $B$ and $\mathscr{B}^* =\{A:A\in 2^P, A\supseteq B\setminus\{B_{\rm max}\}\ {\rm for\ some}\ B\in\mathscr{B}\}$.
\end{lem}

We call $\mathscr{B}$ in Lemma \ref{broken thm} a {\it broken system} of $f$ and $B\setminus\{B_{\max}\}$  a {\it broken set} for any $B\in\mathscr{B}$.

To apply Lemma \ref{broken thm} we need to define a broken system and broken sets for signed graphs. We  follow the idea of the notion of `bonds' introduced in \cite{Chen,Zaslavsky}. For a signed graph $G$ and $X\subseteq V(G)$, denote by $[X, X^C]$ the set of edges between $X$ and its complement $X^C$, by $G[X]$ the subgraph of $G$ induced by $X$, and by $E(X)$ the set of the edges in $G[X]$.  A non-empty edge subset $B\subseteq E(G)$ is called a {\it cut} \cite{Chen} or {\it improving set} \cite{Zaslavsky} of $G$ if it has the form $B=[X, X^C]\cup E_X$, where $X\subseteq V(G)$ is non-empty and $E_X\subseteq E(X)$ is minimal to have $G[X]-E_X$ balanced. A cut is called a {\it bond} of $G$ if it is minimal.
We note that, in the case when $G$ is balanced, we have $E_X=\emptyset$ by the minimality of $E_X$ and therefore, a bond is exactly a usual bond as in an ordinary graph. In this sense, the notion `bond' for signed graph is a very nice extension of that for ordinary graphs \cite{Jin}.

By the definition of the broken set, it is not difficult to see that if $B$ is a bond then, for any $e\in B$,
\begin{equation}
\beta(G-B)=\beta(G-(B\setminus\{e\})).
\end{equation}
On the other hand, by Corollary \ref{excl-incl}, we have
$$F_0(G,x)=\sum_{F\subseteq E}(-1)^{|F|}x^{\beta(G-F)}.$$

Thus, an edge subset of $G$ is a broken set of $F_0(G,x)$ if it has the form $B\setminus\{B_{\max}\}$ for some $B\subseteq E(G)$ such that, for any $A\supseteq B$,
\begin{equation}
\beta(G-A)=\beta(G-(A\setminus\{B_{\max}\})).
\end{equation}

 On the other hand, by (14), for any bond $B$ we have
  $$\beta(G-B)=\beta(G-(B\setminus\{B_{\max}\})).$$

 Moreover, it is not difficult to see that, for any $A\supseteq B$, (15) is satisfied by $A$ and $B$. Thus, $B\setminus\{B_{\max}\}$ is a broken set of $F_0(G,x)$ for any bond $B$ and is called a {\it broken bond} of $G$.

Let $\mathscr{B}$ be the class of all the broken bonds of $G$ and let
$$\mathscr{B}^*=\{F:F\in 2^{E(G)}, F\supseteq B\ {\rm for\ some}\ B\in\mathscr{B}\}.$$

Then by Lemma \ref{broken thm}, we have the following result immediately.
\begin{thm} \label{broken} For any signed graph $G$ with a linear order $\prec$ on  $E(G)$,
\begin{equation}
F_0(G,x)=\sum_{F\in 2^{E(G)}\setminus\mathscr{B}^*}(-1)^{|F|}x^{\beta(G-F)}.
\end{equation}
\end{thm}

\noindent{\bf Remark 2}. If $G$ is balanced, then each broken bond is exactly a usual broken bond of an ordinary  graph. In this case, (16) is still valid. Thus, Theorem \ref{broken} is a generalization of that for ordinary graph \cite{Jin}. Further, in a very special case when an unbalanced signed graph $G$ contains an edge whose removal leaves a balanced graph, the empty set is a broken bond and therefore, any set of edges (including the empty set) contains a broken bond. This case means that $\mathscr{B}^*=2^{E(G)}$ and thus, $F_0(G,x)=0$, which coincides with an obvious fact that such $G$ is not $\Gamma$-flow admissible when $|\Gamma|$ is odd.  \hfill$\square$

\begin{prop}\label{prop} For any signed graph $G$ and $F\subseteq E(G)$, if  $F$ contains no broken bond then each component of $G-F$ is unbalanced, unless $G$ is balanced.
\end{prop}
\begin{proof} To the contrary suppose that one component $\omega$ of $G-F$ is balanced. Let $B=[V(\omega),\overline{V(\omega)}]\cup E_F$, where $E_F$ is the set of edges in $F$ whose two end vertices are both in $\omega$. Then  $B$ is a bond since $\omega$ is balanced and thus $B\setminus\{B_{\max}\}$ is a broken bond. Notice that $B\setminus\{B_{\max}\}\subset B\subseteq F$, which contradicts that $F$ contains no broken bond.
\end{proof}

Let $\sigma(G)$ be the number of those edges $e$ such that there is an edge $e'$ with $e\prec e'$   satisfying one of the following three conditions: \\
1). one of $e$ and $e'$ is a cut edge and $G-\{e,e'\}$ has a balanced component;\\
2). $\{e,e'\}$ is an edge cut and $G-\{e,e'\}$ has a balanced component; \\
3). $\{e,e'\}$ is contained in a component $\omega$ of $G$ and $\omega-\{e,e'\}$ is balanced.

\begin{cor}\label{cor}  Let $G$ be an unbalanced, $\Gamma$-flow admissible ($|\Gamma|$ is odd) signed graph with $n$ vertices and $m$ edges. Then for any linear order $\prec$ on  $E(G)$,
\begin{equation}
F_0(G,x)=a_0x^{m-n}-a_1x^{m-n-1}+a_2x^{m-n-2}-\cdots+(-1)^{m-n}a_{m-n},
\end{equation}
where, for any $i\in\{0,1,\cdots,m-n\}$, $a_i$  is the number of the edge subsets of $G$ having $i$ edges and containing no broken bond as a subset. In particular,\\
1). $a_i>0$ for every $i=0,1,2,\cdots,m-n$;\\
2). $a_0=1$;\\
3). $a_1=m-\sigma(G)$;
\end{cor}
\begin{proof} Let $F\subseteq E(G)$ be an edge subset that contains no broken bond. Since $G$ is unbalanced, so by Proposition \ref{prop}, every component $\omega$ of $G-F$ is unbalanced. Thus,  $\beta(\omega)=m(\omega)-n(\omega)$ due to (11). Therefore,
$$\beta(G-F)=\sum_{\omega}\beta(\omega)=m(G-F)-n(G-F)=m-n-|F|,$$
where the sum is taken over all the components of $G-F$. This equation means that the value of $\beta(G-F)$ is determined uniquely by the number of edges in $F$, as long as $F$ contains no broken bond. So by Theorem \ref{broken},  the coefficient of $(-1)^ix^{m-n-i}$ in $F_0(G,x)$ counts exactly those edge subsets $F$ which have $i$ edges and contain no broken bond. Thus, (17) follows directly.

\noindent 1). We first show that there is an edge set $F$ with $n$ edges that contains no broken bond. By the definition of the broken bond,  an edge set $F$ contains no broken bond if and only if $E(G)\setminus F$ contains at least one edge from each broken bond of $G$. Let $F^*$ be maximum such that $E(G)\setminus F^*$ contains at least one edge from each broken bond of $G$ (such $F^*$ clearly exists because $E(G)\setminus\emptyset$ does). Let $\omega $ be a component of $G-F^*$. Then by Proposition \ref{prop}, $\omega$ contains at least one unbalanced circuit, say $C_u$. We claim that $\omega$ does not contain any other circuit.

Suppose to the contrary that $C$ is a circuit in $\omega$ with $C\not=C_u$. Since $C$ is a circuit,  the property that $G-F^*$  contains at least one edge from each broken bond is still satisfied by $G-F^*-C_{\max}$ because any bond containing $C_{\max}$ must contain another edge $e$ on $C$ with, of course, $e\prec C_{\max}$.  This contradicts our assumption that $F^*$ is  maximum. Our claim follows.

In a word, each component $\omega$ of $G-F^*$ contains exactly one unbalanced circuit and no any other circuit. This means that $m(\omega)=n(\omega)$ and therefore, $m(G-F^*)=n$, i.e., $|F^*|=m-n$. Thus, $a_{m-n}>0$. Further, if an edge subset $F$ contains no broken bond then any subset of $F$ contains neither broken bond, which implies $a_i>0$ for any $i$ with $0\leq i\leq m-n$.

\noindent 2). Since $G$ is flow-admissible, as pointed out in Remark 1, $G$ contains no edge whose removal leaves a balanced graph. This means that the empty set is not a broken bond. Thus, $a_0$ equals the number of the edge subsets of $G$ having $0$ edges, that is, the unique empty set.

\noindent 3). Now we consider the coefficient $a_1$. From the above discussion we see that $a_1$ equals the number of the edges that are not broken bond. On the other hand, an edge $e$ is a broken bond if there is $e'$ such that $B=\{e,e'\}$ is a bond and $e'=B_{\max}$. By the definition of a bond, $B=\{e,e'\}$ must satisfy one of the above three conditions and, vice versa.
\end{proof}

\section{Applications}

An ordinary graph can be viewed as a signed graph that contains no any unbalanced circuit. Oppositely, our first application is to consider a class of the signed graphs which are $\Gamma$-flow admissible for $|\Gamma|$ odd but contain no any balanced circuit.

For a tree $T$, let $G_T$ be the signed graph obtained from $T$ by replacing each of its end vertices (the vertices of degree 1) with an unbalanced circuit. It is clear that $G_T$ contains no balanced circuit.

 Let $v_1,v_2,\cdots,v_p$ be the vertices in $T$ that have degree at least 3 and let $d_1,d_2,\cdots,d_p$ be their degrees, respectively. Choosing an arbitrary vertex $r$ of $T$ as the root, we get a rooted tree (here the `rooted tree' is not the same thing as the `signed rooted tree' defined earlier). For a vertex $v_i$ (with degree at least 3) and an edge $e$ incident with $v_i$, we call $e$ {\it the father} of {\it the family} $v_i$ if $e$ is nearer to the root than other edges incident with $v_i$ and call every edge other than the father a {\it child} of the family $v_i$. In particular, we call the set of all the children of $v_i$ {\it the children class} of $v_i$ and denote it by $C(v_i)$.

 Let $\prec$ be an ordering on $E(G_T)$ such that no child is greater than its father and no edge on an unbalanced circuit is greater than one on $T$.
 Let $F$ be an edge set  of $G_T$ that contains no broken bond.  By Corollary \ref{cor},  $F$ contribute $(-1)^{|F|}x^{m-n-|F|}$ to $F_0(G_T,x)$, where $m=|E(G_T)|,n=|V(G_T)|$. On the other hand, by our definition of $\prec$, $F$ contains no broken bond if and only if $F$ contains neither an edge from an unbalanced circuit nor a children class of a family. For any vertex $v_i$, let $F_i=F\cap C(v_i)$. In particular, let $F_r=F\cap \{e_r\}$, where $e_r$ is the unique edge incident with the root $r$. Thus, the contribution of $F$  to $F_0(G_T,x)$ can be specified as
 \begin{equation}
 x^{m-n}(-1)^{|F_r|}x^{-|F_r|}\prod_{i=1}^p(-1)^{|F_i|}x^{-|F_i|}.
 \end{equation}
  On the other hand, we notice that $m-n=(d_1-2)+(d_2-2)+\cdots+(d_p-2)+1$.  Rewrite (18) as
 $$(-1)^{|F_r|}x^{1-|F_r|}\prod_{i=1}^p(-1)^{|F_i|}x^{d_i-2-|F_i|}.$$
 This means that $(-1)^{|F_r|}x^{1-|F_r|}$ and $(-1)^{|F_i|}x^{d_i-2-|F_i|}$ could be regarded as the contribution of $F$ restricted on $\{e_r\}$ and $C(v_i)$, respectively. Since $F\cap\{e_r\}=\emptyset$ or $F\cap\{e_r\}=\{e_r\}$, all the possible contributions of $F$ restricted on $\{e_r\}$ can be represented as $(-1)^{|\emptyset|}x^{1-|\emptyset|}+(-1)^{|\{e_r\}|}x^{1-|\{e_r\}|}=x-1$.

 In general, for any $v_i$, since $v_i$ has exactly $d_i-1$ children, all the possible contributions of $F$ restricted on $C(v_i)$ equals $$x^{d_i-2}-{{d_i-1}\choose{1}}x^{d_i-3}+\cdots+(-1)^{d_i-2}{{d_i-1}\choose{d_i-2}}.$$
 Thus, the total contributions of all $F$  that contains no broken bond equals
$$F_0(G_T,x)=(x-1)\prod_{i=1}^{p}(x^{d_i-2}-{{d_i-1}\choose{1}}x^{d_i-3}+\cdots+(-1)^{d_i-2}{{d_i-1}\choose{d_i-2}}).$$

Our second application is to show that all the broken bonds in a signed graph form a nice topological structure, namely, the homogeneous simplicial complex.
A finite collection $\mathscr{S}$  of finite sets is called a {\it simplicial complex} if $S\in \mathscr{S}$ implying $T\in \mathscr{S}$ for any $T\subseteq S$. A simplicial complex is {\it homogeneous} \cite{Wilf01} or {\it pure} \cite{Brylawski01} if all  the maximal simplices have the same dimension (cardinality). A classic example of homogeneous simplicial complex related to a graph is  the  broken-circuit complex \cite{Brylawski01,Brylawski02}. It was shown  \cite{Wilf01} that the class $\mathfrak{B}(G)$ consisting of all the edge subsets of an ordinary graph $G$ that contain no broken circuit is a homogeneous simplicial complex of top dimension $|V(G)|-1$ and, moreover, the coefficients of the chromatic polynomial of $G$ are the simplex counts in each dimension of  $\mathfrak{B}(G)$.

Let  $\mathfrak{F}(G)$ be the class consisting of all the edge subsets of a signed graph $G$ that contain no broken bond.
\begin{cor}\label{simplex}  Let $G$ be a non-trivial signed graph with $n$ vertices, $m$ edges and with a linear order $\prec$ on  $E(G)$.  Then \\
1). $\mathfrak{F}(G)$ is a  homogeneous  simplicial complex, i.e., every simplex is a subset of some simplex of top dimension $m-n$;\\
2). An edge set $F$ is a simplex of top dimension $m-n$ of $\mathfrak{F}(G)$ if and only if $E(G)\setminus F$ contains at least one edge from each broken bond of $G$ and each component $G-F$ contains no but exactly one unbalanced circuit;\\
3). For each $i\in\{0,1,2,\cdots,m-n\}$, the coefficient $a_i$ in $F_0(G,x)$ is the number of the $i$-dimensional simplexes in $\mathfrak{F}(G)$.
\end{cor}
\begin{proof} 1). It is obvious that $\mathfrak{F}(G)$ is a simplicial complex. We prove that  $\mathfrak{F}(G)$ is homogeneous.

Let $F$ be a set of edges that contains no broken bond. If $|F|=m-n$ then we are done. We now assume that $|F|<m-n$, i.e., $|E(G-F)|>n$. In this case, it can be seen that there is a component $\omega$ in $G-F$ which contains at least two circuits $C$ and $C'$. By Proposition \ref{prop}, one of these two circuits, say $C$, is unbalanced. So by the same argument as that in Corollary \ref{cor}, we can find an edge $e$ in $C'$ such that $G-F-e$ still contains an edge from each broken bond. Replacing $F$ by $F\cup\{e\}$, the assertion follows by repeating  this procedure, until $|F|=m-n$.

2) and 3) follows directly by  Corollary  \ref{cor}.   \end{proof}

\section*{Acknowledgments}
This work was supported by the National Natural Science Foundation of China [Grant numbers, 11471273, 11561058].

\end{document}